\pdfoutput=1
\documentclass[11pt]{amsart}
\usepackage{amssymb}
\usepackage{multirow}
\usepackage{mathtools}
\usepackage{hyperref}
\usepackage{tikz-cd}
\usepackage{graphicx}
\usepackage{color}

\usepackage[margin=1in]{geometry}

\newtheorem{thm}{Theorem}[section]
\newtheorem{cor}[thm]{Corollary}
\newtheorem{prop}[thm]{Proposition}
\newtheorem{lem}[thm]{Lemma}

\theoremstyle{definition}
\newtheorem{defn}[thm]{Definition}

\newtheorem{question}[thm]{Question}

\theoremstyle{remark}
\newtheorem{rem}[thm]{Remark}

\newcommand{\Q}{\mathbb{Q}}
\newcommand{\Z}{\mathbb{Z}}
\newcommand{\C}{\mathbb{C}}

\DeclareMathOperator{\End}{End}

    \newtheoremstyle{TheoremNum}
        {7pt}{7pt}              
        {\itshape}                      
        {}                              
        {\bfseries}                     
        {.}                             
        { }                             
        {\thmname{#1}\thmnote{ \bfseries #3}}
    \theoremstyle{TheoremNum}

\begin{document}
\title{Monodromy of Kodaira Fibrations of Genus $3$}
\author{Laure Flapan}
\address{Department of Mathematics, Massachusetts Institute of Technology, Cambridge, MA 02142}
\email{lflapan@mit.edu}


\begin{abstract}
A Kodaira fibration is a non-isotrivial fibration $f\colon S\rightarrow B$ from a smooth algebraic surface $S$ to a smooth algebraic curve $B$ such that all fibers are smooth algebraic curves of genus $g$. Such fibrations arise as complete curves inside the moduli space $\mathcal{M}_g$ of genus $g$ algebraic curves. We investigate here the possible connected monodromy groups of a Kodaira fibration in the case $g=3$ and classify which such groups can arise from a Kodaira fibration obtained as a general complete intersection curve inside a subvariety of $\mathcal{M}_3$ parametrizing curves whose Jacobians have extra endomorphisms.
\end{abstract}
\maketitle

\section{Introduction}\label{intro}
A Kodaira fibration is a non-isotrivial fibration $f\colon S\rightarrow B$ from a smooth algebraic surface $S$ to a smooth algebraic curve $B$ such that for every $b\in B$, the fiber $F_b:=f^{-1}(b)$ is also a smooth algebraic curve.  The non-isotriviality of the fibration, meaning the property that not all fibers $F_b$ are isomorphic as algebraic varieties, ensures that the fundamental group $\pi_1(B)$ does not act trivially on the fibers. 

Such fibrations were originally constructed by Kodaira \cite{kodsur} as a way to show that, unlike the topological Euler characteristic, the signature $\sigma$ of a manifold, meaning the signature of the intersection form on the middle homology of the manifold, is not multiplicative for fiber bundles. More precisely, for any fiber bundle $\phi\colon X \rightarrow Y$ with fiber $F_y$, the topological Euler characteristic $\chi_{\mathrm{top}}$ satisfies $\chi_{\mathrm{top}}(X)=\chi_{\mathrm{top}}(Y)\chi_{\mathrm{top}}(F_y).$ Prior to Kodaira's construction, Chern-Hirzebruch-Serre \cite{chs} had shown that if $\pi_1(Y)$ acts trivially on the fiber $F_y$, then the signature also satisfies $\sigma(X)=\sigma(Y)\sigma(F)$. Kodaira's construction proved that this hypothesis about the fundamental group was necessary, since for any Kodaira fibration $f\colon S\rightarrow B$, the surface $S$ has signature $\sigma(S)>0$, while the algebraic curves $F_b$ and $B$ have antisymmetric intersection form and thus satisfy $\sigma(F_b)=\sigma(B)=0$. 

Any Kodaira fibration  $f\colon S\rightarrow B$ induces a short exact sequence of fundamental groups
\begin{equation}\label{fundgps}
1\rightarrow \pi_1(F_b)\rightarrow \pi_1(S)\rightarrow \pi_1(B)\rightarrow 1.
\end{equation}

In fact, Kotschick shows in \cite[Proposition 1]{kotschick} that any compact complex surface $S$ whose fundamental group fits into an exact sequence of the form \eqref{fundgps} satisfies $\chi_{\mathrm{top}}(S)=\chi_{\mathrm{top}}(B)\chi_{\mathrm{top}}(F_b)$ if and only if the sequence \eqref{fundgps} is induced by a Kodaira fibration $S\rightarrow B$. It is thus natural to ask: 
\begin{question}\label{fund groups question}
For which extensions 
\[1\rightarrow \pi_1(F_b)\rightarrow G\rightarrow \pi_1(B)\rightarrow 1\]
 as in \eqref{fundgps}, is the group $G$ the fundamental group of a Kodaira surface?
 \end{question}
 
Note that Kodaira in \cite{kodsur} together with Kas in \cite{kas} show that a Kodaira fibration must have base curve $B$ of genus at least $2$ and have fibers $F_b$ of genus at least $3$. The short exact sequence (\ref{fundgps}) induces a homomorphism
\[\rho: \pi_1(B)\rightarrow \mathrm{Mod}(F_b)\subset O(\pi_1(F_b))\]
into the mapping class group of $F_b$. Because the genus $g$ of $F_b$ is at least $3$, the center $Z(\pi_1(F_b))$ is trivial \cite{FM}. Hence, since extensions with outer action  $\rho$ are parametrized by $H^2(\pi_1(B),  Z(\pi_1(F_b)))$, it follows that the homomorphism $\rho$ completely determines the extension (\ref{fundgps}) \cite[Corollary 6.8]{brown}.

Letting $V=H_1(F_b,\mathbb{Z})\otimes \mathbb{Q}$, the homomorphism $\rho$ induces a homomorphism
\[\overline{\rho}: \pi_1(B)\rightarrow GL(V).\]

Observe that the Kodaira fibration $f\colon S\rightarrow B$ determines a map $B\rightarrow \mathcal{M}_{g}$ to the moduli space of curves of genus $g$, sending a point $b\in B$ to the curve $F_b$. Since, by assumption, the fibration $f$ is non-isotrivial, the map $B\rightarrow \mathcal{M}_{g}$ has $1$-dimensional image. It follows from the Torelli theorem \cite[Chapter IV.3]{arbarello} that the induced map to the moduli space of principally polarized abelian varieties $B\rightarrow \mathcal{A}_{g}$, sending $b\in B$ to $H_1(F_b,\mathbb{Z})$, has one-dimensional image. 

Consider the variation of $\Q$-Hodge structures $R_1f_*\mathbb{Q}$. We say that a variation of Hodge structures $\mathbb{V}$ over a connected algebraic variety $Y$ is \emph{isotrivial} if there is some base change $Y'$ of $Y$ on which the induced variation of Hodge structures is constant. Schmid's rigidity theorem \cite[Theorem 7.24]{schmid} implies that if the monodromy representation $\overline{\rho}$ had finite image, then $R_1f_*\mathbb{Q}$ would be isotrivial. However since the map $B\rightarrow \mathcal{A}_{g}$ has one-dimensional image, the variation of of Hodge structures $R_1f_*\mathbb{Q}$ is not isotrivial. Therefore, the representation $\overline{\rho}$, and consequently the homomorphism $\rho$ as well, must have infinite image. 

To better understand this image, note that any variation of $\Q$-Hodge structures $\mathbb{V}$ over a connected algebraic variety $Y$ will yield a monodromy respresentation $\Phi\colon \pi_1(Y,y)\rightarrow GL(V_y)$ for $y\in Y$. 

\begin{defn} The \emph{connected monodromy group} $T$ of the variation $\mathbb{V}$ is the connected component of the identity of the smallest $\mathbb{Q}$-algebraic subgroup of $GL(V_y)$ containing the image of $ \pi_1(Y,y)$.
\end{defn}

Thus one way to approximate Question \ref{fund groups question} about which groups can arise as the fundamental group of a Kodaira surface is to ask: 
\begin{question}\label{monodromy question}
Which groups can arise as the connected monodromy group of a Kodaira fibration?
\end{question}

This question was first studied in \cite[Section 5.2]{arapura} as part of a larger study of restrictions on possible connected monodromy groups of fibered projective varieties. In that paper, from which this article draws many ideas, Arapura proves, for instance,  that the connected monodromy group of a Kodaira fibration must always be a non-trivial semisimple group of symplectic Hermitian type \cite[Lemma 5.5]{arapura}.

The main result of this article, Theorem \ref{main theorem}, provides an answer to Question \ref{monodromy question} in the case of a Kodaira fibration whose fibers have genus $3$. The theorem gives precise characterizations of the possible connected monodromy groups as $\Q$-algebraic groups and their representations. After base changing to $\mathbb{C}$, Theorem \ref{main theorem} yields

\begin{thm}\label{complex theorem}
The connected monodromy group of a genus $3$ Kodaira fibration must be isomorphic over $\C$ to one of the following groups:
\begin{enumerate}
\item
\begin{enumerate}
\item \label{t1}$Sp(6)$
\item \label{t2}$SL(2)\times SL(2)\times SL(2) $
\item \label{t3}$SU(3) $
\end{enumerate}
\item 
\begin{enumerate}
\item \label{t4}$Sp(4)$
\item \label{t5}$SL(2)\times SL(2)$
\end{enumerate}
\end{enumerate}
Moreover, the groups \eqref{t1}, \eqref{t2}, \eqref{t4} all arise from  Kodaira fibrations obtained as general complete intersection curves in a subvariety of $\mathcal{A}_3$ whose points all have endomorphisms by a specified ring.  By contrast,  the groups \eqref{t3} and \eqref{t5} are not known to actually occur. 
\end{thm}

Theorem \ref{main theorem} thus shows that while the options for admissible extensions in \eqref{fundgps} are quite limited, already in genus $3$ there are Kodaira fibrations that arise from extensions beyond the most general one. Phrased differently, there exist genus $3$ Kodaira fibrations $f\colon S\rightarrow B$, corresponding to complete curves in $\mathcal{M}_3$, such that the Jacobian of the fibers comes equipped with additional endomorphisms beyond just $\Z$. This stands in contrast with Gonz\'{a}lez-D{\'{\i}}ez and Harvey's result in \cite[Theorem 2]{gonzalez} that there are no complete curves inside the locus of $\mathcal{M}_3$ parametrizing curves with extra automorphisms. Moreover, the only known explicit examples of complete curves in $\mathcal{M}_3$ are due to a construction of Zaal, who proves that the Jacobian of the curve corresponding to a very general fiber in his examples has no extra endomorphisms \cite{zaal}.

The main strategy of the proof of Theorem \ref{main theorem} is to study the possible endomorphism algebras of the $\Q$-Hodge structure $V=H_1(F_b,\Q)$, for $F_b$ a very general fiber of a genus $3$ Kodaira fibration. Recall here that a $\Q$-Hodge structure $W$ of weight $w$ is a $\Q$-vector space equipped with a decomposition of $\C$-vector spaces $W\otimes_\Q\C=\bigoplus_{p+q=w}W^{p,q}$ such that the complex conjugate $\overline{W}^{p,q}$ is equal to $W^{q,p}$. Thus $V$ is a $\Q$-Hodge structure of weight $-1$ in that it has decomposition $V\otimes_\Q\C=V^{-1,0}\oplus V^{0,-1}$, where $V^{p,q}\cong H^{-q}(F_b,\Omega_{F_b}^{-p})$ and $\Omega_{F_b}^{-p}$ denotes the sheaf of holomorphic $-p$-forms on the curve $F_b$. A polarization of a $\Q$-Hodge structure $W$ of weight $w$ is an alternating (respectively symmetric) form $\langle,\rangle\colon W\times W\rightarrow \Q$  if $w$ is odd (respectively even) such that $\langle W^{p,q},W^{p',q'}\rangle =0$ if $p'\ne w-p$ and $i^{p-q}(-1)^{\frac{w(w-1)}{2}}\langle x,\overline{x}\rangle>0$  for all nonzero $x\in W^{p,q}$. Thus in particular, the $\Q$-Hodge structure $V$ comes equipped with a canonical polarization corresponding to the intersection form.

Throughout the paper, we make use of the equivalence of categories, discussed in greater detail in Section \ref{background}, between the category of complex abelian varieties up to isogeny and that of polarizable $\Q$-Hodge structures of type $\{(-1,0), (0,-1)\}$, meaning $\Q$-Hodge structures $W$ with decomposition $W\otimes_\Q\C=W^{-1,0}\oplus W^{0,-1}$.  This equivalence allows us to go back and forth between the $\Q$-Hodge structure $V=H_1(F_b,\Q)$ and the Jacobian, considered up to isogeny, of a very general fiber $F_b$ of a Kodaira fibration $f\colon S\rightarrow B$. 

Determining the possible endomorphism algebras of $V=H_1(F_b,\Q)$ thus amounts to studying subvarieties of $\mathcal{A}_3$ parametrizing abelian varieties with fixed endomorphisms and determining which of these may contain a complete curve corresponding to a family of Jacobians of smooth curves. This problem is much more tractable in genus $3$, the minimal possible genus for a Kodaira fibration, than in higher genus due to the Schottky problem. Namely, in general, it is very difficult to describe the locus of abelian varieties of dimension $g$ corresponding to Jacobians of smooth genus $g$ curves. In dimension $3$ however, the only elements of $\mathcal{A}_3$ which are not Jacobians of smooth genus $3$ curves are those lying in the decomposable locus $\mathcal{A}_3^{\mathrm{dec}}$ consisting of principally polarized abelian varieties that can be written as a product of smaller dimensional principally polarized abelian varieties.

The classification of the possible endomorphism algebras of $V=H_1(F_b,\Q)$  in the case of a genus $3$ Kodaira fibration is accomplished in Propositions \ref{simple case} and \ref{prop non-simple case}. Crucial is the classification by Albert \cite{albert} and Shimura \cite{shimura} of the possible endomorphism algebras of a $g$-dimensional abelian variety.  Proposition \ref{simple case} treats the case where the Hodge structure $V=H_1(F_b,\Q)$ is simple, in the sense that it cannot be written as a sum of smaller-dimensional $\Q$-Hodge structures. Proposition \ref{prop non-simple case} treats the more complicated case that $V=H_1(F_b,\Q)$ is non-simple. In this case, since $V$ is non-simple, its corresponding abelian variety is isogenous to a point in $\mathcal{A}_3^{\mathrm{dec}}$. However, since this abelian variety must correspond to the Jacobian of a smooth curve, it cannot be isomorphic to a point in $\mathcal{A}_3^{\mathrm{dec}}$. Thus to address the non-simple case, a careful study of isogenies of $3$-dimensional abelian varieties is required (see Sections \ref{hecke section} and \ref{decomposable isogeny section}). Roughly speaking, we use the machinery of Hecke correspondences to determine when the relevant locus of abelian varieties with given endomorphism algebra may be moved away from the decomposable locus in sufficiently high codimension in order to show there is a complete curve corresponding to a Kodaira fibration contained in this locus. 

Once the classification of possible endomorphism algebras is accomplished, the result of Theorem \ref{main theorem} follows directly using results of Andr\'{e} \cite{andre} and Ribet \cite{ribet1} about Hodge groups, which we explain below. The \textit{Hodge group} $Hg(W)$ of a polarizable $\Q$-Hodge structure $W$ (also called the \textit{special Mumford-Tate group}) is the connected algebraic $\Q$-subgroup of $SL(W)$ whose invariants for any $m,n\in \Z_{\ge0}$ in 
\[\mathcal{T}^{m,n}\coloneqq W^{\otimes m} \otimes (W^*)^{\otimes n}\]
are the Hodge cycles, meaning they have type $p=q$ in the induced Hodge decomposition of $\mathcal{T}^{m,n}$. Andr\'{e}'s theorem (originally stated for mixed $\Q$-Hodge structures, but which we state here only for variations of pure $\Q$-Hodge structures) yields:
\begin{thm}(Andr\'{e} \cite[Theorem 1]{andre})\label{andrethm} Let $\mathbb{V}$ be a polarizable variation of $\Q$-Hodge structures over a smooth connected algebraic variety $Y$. Then for very general $y\in Y$, the connected monodromy group $T$ of $\mathbb{V}$ satisfies
$T\triangleleft [Hg(V_y),Hg(V_y)].$
\end{thm}
Additionally, Ribet shows in  \cite{ribet1} that the Hodge group of a simple abelian variety $A$ of prime dimension is completely determined by $\End_\Q(A)$ and his calculations of $Hg(A)$ yield that $[Hg(A),Hg(A)]$ is always a simple group. The two results together can thus be used to completely determine the possible connected monodromy groups $T$ of a genus $3$ Kodaira fibration once the possible endomorphism algebras of $V=H_1(F_b,\Q)$ are known, which then yields Theorem \ref{main theorem}.

The organization of the paper is as follows. In Section \ref{general kodaira section} we detail the construction of Kodaira fibrations from general complete intersection curves. Section \ref{background} then gives background on $\Q$-Hodge structures of type $\{(-1,0),(0,-1)\}$ and their endomorphism algebras. In Section \ref{section restriction} we give some initial restrictions on the $\Q$-Hodge structure $V=H_1(F_b,\Q)$ on a very general fiber $F_b$ of a Kodaira fibration of arbitrary genus. Section \ref{section simple case} then characterizes the possible endomorphism algebras of $V$ for a genus $3$ Kodaira fibration in the case when $V$ is simple. Section \ref{hecke section} introduces Hecke correspondences and polarized isogenies. Sectin \ref{decomposable isogeny section} then makes use of the machinery of Section \ref{hecke section} to study some sub-loci of the decomposable locus $\mathcal{A}_3^{\mathrm{dec}}$ and their behavior under Hecke translation. Section \ref{non-simple section} then analyzes the possible endomorphism algebras of $V$ when $V$ is not simple.  Lastly, Section \ref{section main result} presents the proof of Theorem \ref{main theorem}.

\section{General Complete Intersection Kodaira Fibrations}\label{general kodaira section}
Let $\mathcal{M}_{g}[n]$ be the fine moduli space of curves of genus $g$ with fixed level $n\ge 3$ structure. Note that the choice of $n$ is to ensure that the moduli space is fine so that a curve in $\mathcal{M}_{g}[n]$ really does parametrize a family of genus $g$ curves. Thus any Kodaira fibration $f\colon S\rightarrow B$ corresponds, up to possible base change, to a complete curve in $\mathcal{M}_{g}[n]$. Let $\mathcal{A}_g[n]$ denote the fine moduli space of principally polarized abelian varieties of genus $g$ with level $n$ structure. The Torelli morphism
\[J\colon \mathcal{M}_g[n]\rightarrow \mathcal{A}_g[n]\]
sends a curve to its Jacobian together with its canonical polarization and level structure. The closure $\mathcal{T}_g[n]$ of the image $J(\mathcal{M}_g[n])$  of $\mathcal{M}_g[n]$ in $\mathcal{A}_g[n]$ is called the \textit{Torelli locus}. 
The boundary $T_g^{\mathrm{dec}}[n]\coloneqq T_g[n]\backslash J(\mathcal{M}_g[n])$ consists of decomposable principally polarized abelian varieties, meaning principally polarized abelian varieties that can be written as the product of two smaller dimensional principally polarized abelian varieties (all with level structure $n$).  

Although the Torelli morphism on the coarse moduli space $\mathcal{M}_g$ of curves of genus $g$ (without specifying a level structure) is an embedding, for $g>2$ the Torelli morphism $J\colon \mathcal{M}_g[n]\rightarrow \mathcal{A}_g[n]$ is ramified at the hyperelliptic locus $\mathcal{H}_g[n]$ of $\mathcal{M}_g[n]$. Outside of $\mathcal{H}_g[n]$, the morphism $J$ is an immersion. The hyperelliptic locus $\mathcal{H}_g[n]$ has codimension in $\mathcal{M}_g[n]$ at least $2$ for $g\ge 4$ and has codimension $1$ for $g=3$. 

Neither of the moduli spaces $\mathcal{M}_{g}[n]$, $\mathcal{A}_g[n]$ is complete, however $\mathcal{A}_g[n]$ comes equipped with a minimal compactification $\mathcal{A}_g[n]^*$ called the Satake-Baily-Borel compactification. The Satake-Baily-Borel compactification $\mathcal{M}_{g}[n]^*$ is defined by taking the closure of $J(\mathcal{M}_{g}[n])$ inside $\mathcal{A}_g[n]^*$. For $g>2$, the boundary $ \mathcal{M}_{g}[n]^*\backslash  \mathcal{M}_{g}[n]$ has codimension $2$ inside $\mathcal{M}_{g}[n]^*$.

The following argument to produce Kodaira fibrations whose fibers have Jacobian lying in some subvariety of $\mathcal{A}_g[n]$ is obtained from a standard argument, which can be found for instance in \cite[Section 1.2.1]{catanese} and which is used to prove the existence of Kodaira fibrations more generally.

\begin{prop}\label{complete intersection kod fib} Suppose $Z$ is a subvariety of $\mathcal{A}_g[n]$, for $n,g\ge 3$, such that
\begin{enumerate}
\item $Z\cap \mathcal{T}_g[n]^{\mathrm{dec}}$ has codimension at least $2$ in $Z\cap \mathcal{T}_g[n]$
\item $(Z\cap \mathcal{T}_g[n])^*$ has boundary $(Z\cap \mathcal{T}_g[n])^*\backslash (Z\cap \mathcal{T}_g[n])$ of codimension at least $2$ in $(Z\cap \mathcal{T}_g[n])^*$.
\end{enumerate}
Then there exists a complete curve $C$ in $Z\cap \mathcal{T}_g[n]$ obtained as an intersection of general ample divisors on $Z\cap \mathcal{T}_g[n]$ such that, after possibly taking a ramified cover, the curve $C$ corresponds to a genus $g$ Kodaira fibration $f\colon S\rightarrow B$. 
\end{prop}

\begin{proof}
Recall that $T_g[n]^{\mathrm{dec}}$ is defined as the boundary $T_g[n]\backslash J(M_g[n])$. Thus if $Z\cap T_g[n]^{\mathrm{dec}}$ has codimension at least $2$ in $Z\cap \mathcal{T}_g[n]$, then a general complete intersection curve in $Z\cap \mathcal{T}_g[n]$, meaning one obtained as an intersection of general ample divisors, will be contained in $J(M_g[n])$. Moreover, since the Satake compactification $(Z\cap \mathcal{T}_g[n])^*$ has boundary $(Z\cap \mathcal{T}_g[n])^*\backslash (Z\cap \mathcal{T}_g[n])$ of codimension at least $2$ inside $(Z\cap \mathcal{T}_g[n])^*$ such a curve $C$ will be complete. So we indeed have a complete curve $C$ lying inside $J(\mathcal{M}_g[n]) \cap Z$. The issue is to show that the curve $C$ actually parametrizes a family of curves in $\mathcal{M}_{g}[n]$.

Now, the hyperelliptic locus $\mathcal{H}_g[n]$ is affine and has codimension $\ge 1$ in $\mathcal{M}_g[n]$ for $g\ge 3$. In particular, since the curve $C$ is complete, it cannot be contained in the hyperelliptic locus in $\mathcal{A}_g$, and consequently neither can $Z$. So we know $J(\mathcal{H}_g[n])\cap Z$ has codimension at least $1$ in $Z$ and that, although the curve $C$ cannot necessarily be chosen to avoid the hyperelliptic locus, it can be chosen to intersect the hyperelliptic locus transversally at general points $P_1,\ldots, P_k$. Since the Torelli map $J$ is not an immersion above these points, we do not have a family of curves over $C$, only over $C-\{P_1,\ldots,P_k\}$.  Now take a double cover $\pi\colon C'\rightarrow C$
branched over a set of points that includes $\{P_1,\ldots, P_k\}$. For each $\le i\le k$, let $P_i'\coloneqq \pi^{-1}(P_i)$. Then the Kuranishi family of each $
P_i$ has a map to $\mathcal{M}_g[n]$ ramified over the hyperelliptic locus. Namely, the family of curves over $C'\backslash \pi^{-1}(\{P_1,\ldots, P_k\})$ extends to a family of genus $g$ curves over all of $C'$.

\end{proof}

\begin{defn}\label{definition general complete intersection} A Kodaira fibration is said to be \textit{general complete intersection} if it is obtained, as in the statement of Proposition \ref{complete intersection kod fib}, from a general complete intersection curve in  $Z\cap \mathcal{T}_g[n]$, for $Z$ a subvariety of $\mathcal{A}_g[n]$.
\end{defn}

\begin{rem}\label{complete intersection remark} Proposition \ref{complete intersection kod fib} shows in particular that there exist general complete intersection Kodaira fibrations for every genus $g\ge 3$ arising as general complete intersection curves in $J(M_g[n])$. The Jacobian of a very general fiber of such a fibration will thus be a general genus $g$ abelian variety, meaning one with endomorphisms only by $\Z$.

More generally, letting $Z$ be the subvariety of $\mathcal{A}_g[n]$ cut out by principally polarized abelian varieties with specified endomorphism ring, Proposition \ref{complete intersection kod fib} gives conditions under which one may show the existence of a Kodaira fibration such that the Jacobian of a very general fiber comes equipped with endomorphisms by this specified ring. 
\end{rem}


\section{Mumford-Tate Domains of $\Q$-Hodge structures with Extra Endomorphisms}\label{background}
Let $V$ be a simple polarized $\mathbb{Q}$-Hodge structure of type $\{(-1,0),(0,-1)\}$ with polarization denoted by $\langle, \rangle$. The category of abelian varieties over $\mathbb{C}$ up to isogeny is equivalent to the category of polarizable $\mathbb{Q}$-Hodge structures of type $\{(-1,0),(0,-1)\}$ via the functor that sends an abelian variety $A$ to its homology $H_1(A,\mathbb{Q})$ \cite[Theorem 2.2]{moonenmt}. This enables us to apply results originally formulated in the language of abelian varieties to $\Q$-Hodge structures of this type. 

The endomorphism algebra $L$ of the Hodge structure $V$ is a division algebra over $\mathbb{Q}$ with an involution $a\to \overline{a}$ for each $a\in L$, given by $\langle ax,y\rangle=\langle x, \overline{a}y\rangle$ for any $x,y \in V$. This involution is called the \emph{Rosati involution}. 

The Rosati involution on $L$ is a \emph{positive involution}, meaning that $L$ is finite-dimensional as a $\Q$-vector space and that the reduced trace $\mathrm{tr}^L_\mathbb{Q}(x\overline{x})$ is positive as an element of $\mathbb{R}$ for all nonzero $x$ in $V$ \cite[Remark 1.20]{moonenmt}. It follows that if $L$ is a field, then $L$ is either a totally real or a CM field and hence the Rosati involution on $L$ just corresponds to complex conjugation.

Now let $F_0$ be the center of $L$ and let $F$ be the subfield of $F_0$ fixed by the Rosati involution. Then Albert's classification of division algebras over a number field that have positive involution \cite[Chapter X, \S11]{albert} yields that $L$ must one of the following four types:
\begin{enumerate}
\item Type I: $L=F$ is totally real

\item Type II: $L$ is a totally indefinite quaternion algebra over the totally real field $F$

\item Type III: $L$ is a totally definite quaternion algebra over  the totally real field $F$

\item Type IV: $L$ is a central simple algebra over the CM field $F_0$
\end{enumerate}

Note that since $V$ is a vector space over the division algebra $L$, the degree of $L$ over $\mathbb{Q}$ must divide the dimension of $V$ as a $\mathbb{Q}$-vector space. Hence if $V$ has dimension $2n$, then we may write $2n=m[L:\mathbb{Q}]$ for some positive integer $m$. 

 Writing $l=[F:\mathbb{Q}]$ and $q^2=[L:F_0]$, we have that $V\otimes _\mathbb{Q}\mathbb{C}=V^{-1,0}\oplus V^{0,-1}$ is free of rank $2n/l$ over $F\otimes_\mathbb{Q}\mathbb{C}=\prod_{\sigma\in \Sigma(F)}\mathbb{C}$, where $\Sigma(F)$ denotes the set of embeddings of $F$ into $\mathbb{C}$. Since the action of $F\otimes_\mathbb{Q}\mathbb{C}$ respects the Hodge decomposition of $V\otimes _\mathbb{Q}\mathbb{C}$, any embedding $\sigma\in \Sigma(F)$ acts with the same multiplicity in $V^{-1,0}$ and $ V^{0,-1}$. Hence $2l$ must divide $2n$, meaning that $l$ divides $n$.   In particular, if $L$ is of Type I, then $[L:\mathbb{Q}]$ divides $n$ and so $m$ is even. 
 
If $L$ is of Type IV, we have $L\otimes _\mathbb{Q}\mathbb{C}\cong (M_q(\mathbb{C}))^{2l}$, so let $\chi_1,\ldots, \chi_l,\overline{\chi}_1,\ldots, \overline{\chi}_l$ denote the simple $L\otimes _\mathbb{Q}\mathbb{C}$-modules, each having dimension $q$. For $1\le \nu\le l$, let $r_\nu$ and $s_\nu$ denote the multiplicities of $\chi_\nu$ and $\overline{\chi}_\nu$ respectively of the representation of $F_0$ on $V^{-1,0}\subset V\otimes_\mathbb{Q}\mathbb{C}$. So then for each $1\le \nu\le l$, we know $r_\nu+s_\nu=mq$.

\subsection{Hermitian Symmetric Domains}\label{Hermitian symmetric}
Following \cite{classperiod}, if $D$ is a Hermitian symmetric domain, let $\mathrm{Hol}(D)$ be the group of automorphisms of $D$ as a complex manifold and let $\mathrm{Hol}(D)^+$ be the connected component of $\mathrm{Hol}(D)$ containing the identity. Then letting $\mathfrak{d}=\mathrm{Lie}(\mathrm{Hol}(D)^+)$, there is a unique connected adjoint real algebraic subgroup $G$ of $GL(\mathfrak{d})$ such that inside of $GL(\mathfrak{d})$  we have $G(\mathbb{R})^+=\mathrm{Hol}(D)^+$, where $G(\mathbb{R})^+$ denotes the connected component of the identity of $G(\mathbb{R})$ with respect to the real topology \cite[Proposition 1.7]{milne}. Moreover, isomorphism classes of irreducible Hermitian symmetric domains are classified by the special nodes of the connected Dynkin diagrams associated to the groups $G_\mathbb{C}$, where a node $i$ is said to be special if its corresponding simple root $\alpha_i$ is such that, in the expression of the highest root $\tilde{\alpha}=\sum_{i\in I}n_i\alpha_i$ of the Dynkin diagram of $G_\mathbb{C}$, the coefficient $n_i$ of $\alpha_i$ is equal to $1$. 

Now consider the following Hermitian symmetric domains:
\begin{equation*}
\begin{aligned}
\mathcal{D}^1_{r,s}&=\{z\in M_{r,s}(\mathbb{C})\mid 1-z^t\overline{z} \text{ is positive Hermitian}\}\\
\mathcal{D}^2_r&=\{z\in M_r(\mathbb{C})\mid z^t=-z \text{ and } 1-z^t\overline{z} \text{ is positive Hermitian}\}\\
\mathcal{D}^3_r&=\{z\in M_r(\mathbb{C})\mid z^t=z \text{ and } \mathrm{Im}(z) \text{ is positive symmetric}\}
\end{aligned}
\end{equation*}
In the classification of Hermitian symmetric domains these correspond to types $I_{r,s}$, $II_r$ and $III_r$ respectively, which have associated Dynkin diagrams of types $A_{r+s-1}$, $D_r$, and $C_r$ respectively.

The Hermitian symmetric domain $\mathcal{D}^3_r$ is the Siegel upper half-space of degree $r$ and parametrizes $2r$-dimensional polarizable $\mathbb{Q}$-Hodge structures of type $\{(-1,0),(0,-1)\}$. Hence suppose that $V\in\mathcal{D}^3_n$ is one such Hodge structure and let $L$ be the endomorphism algebra of $V$. We will call the subdomain of $\mathcal{D}^3_n$ parametrizing $2n$-dimensional $\Q$-Hodge structures of type $\{(-1,0),(0,-1)\}$ with endomorphisms by $L$ the \emph{Mumford-Tate domain} of $V$. A more precise definition and a more detailed discussion of Mumford-Tate domains can be found in \cite{domain}.

Recall the notation for $l$, $m$, $r_i$, and $s_i$ from the beginning of this section. Then, following Shimura \cite[Section 2.6]{shimura}, the connected component through $V$ of the Mumford-Tate domain of $\Q$-Hodge structures of type $\{(-1,0),(0,-1)\}$ with endomorphisms by $L$ are the following products of $l$ factors, depending on the Albert classification of $L$:
\begin{enumerate}

\item Type I: $\mathcal{D}^3_{m/2}\times \cdots \times \mathcal{D}^3_{m/2}$

\item Type II: $\mathcal{D}^3_{m}\times \cdots \times \mathcal{D}^3_{m}$

\item Type III: $\mathcal{D}^2_{m}\times \cdots \times \mathcal{D}^2_{m}$

\item Type IV: $\mathcal{D}^1_{r_1,s_1}\times \cdots \times \mathcal{D}^1_{r_l,s_l},$
\end{enumerate}

The corresponding dimensions $d$ of these domains are thus

\begin{enumerate}
\item Type I: $d=\frac{1}{2}\frac{m}{2}\left(\frac{m}{2}+1\right)l$

\item Type II:  $d=\frac{1}{2}m\left(m+1\right)l$

\item Type III: $d=\frac{1}{2}m\left(m-1\right)l$

\item Type IV: $d=\sum_{\nu=1}^l r_\nu s_\nu.$

\end{enumerate}

\section{Restrictions on possible Kodaira fibrations}\label{section restriction}
Let $f\colon S\rightarrow B$ be a Kodaira fibration with fibers of genus $g$. We use the above discussion of Mumford-Tate domains of $\Q$-Hodge structures of type $\{(-1,0),(0,-1)\}$ with endomorphisms by a fixed algebra $L$ to determine some initial restrictions on the $\Q$-Hodge structure $V=H_1(F_b,\Q)$, for $F_b$ a very general fiber of $f$. 

\begin{lem}\label{VHS lemma}
Suppose $\mathbb{V}$ is a polarized variation of $\mathbb{Q}$-Hodge structures over a non-singular variety $B$. Let $V_b$ be a very general fiber and consider the decompostition of $V_b$ as a sum of simple non-isomorphic $\Q$-Hodge structures
\[V_b=\oplus_{i=1}^k V_i^{n_i}.\]
Then there exists a corresponding decomposition
\[\mathbb{V}=\oplus_{i=1}^k \mathbb{V}_i^{n_i}\]
such that $(\mathbb{V}_i)_b=V_i$ for every $1\le i\le k$.
\end{lem}
\begin{proof}
Consider the induced polarized $\Q$-Hodge structure on $\End _\Q(V_b,V_b)$. Then each of the projection maps $p_i\colon V_b\rightarrow V_i$ is a Hodge cycle of type $(0,0)$ in $\End _\Q(V_b,V_b)\cong V_b\otimes V_b^*$. However for $b$ very general, the Hodge cycles of type $(0,0)$ in $V_b\otimes V_b^*$ are locally constant and define a local system on $B$ underlying a polarized variation of $\mathbb{Q}$-Hodge structures of type $(0,0)$ (see the proof of Proposition 7.5 in \cite{deligne}). In particular, the Hodge cycle $p_i$ in $V_b\otimes V_b^*$ lifts to a section $\mathcal{P}_i$ of $\End (\mathbb{V},\mathbb{V})$ and defining $\mathbb{V}_i=\mathcal{P}_i(V)$ yields a decomposition of polarized variations of $\mathbb{Q}$-Hodge structures $\mathbb{V}=\oplus_{i=1}^k \mathbb{V}_i^{n_i}$ with the desired property. 
\end{proof}

\begin{lem}\label{1dim}Consider the polarized variation of $\Q$-Hodge structures $R_1f_*\mathbb{Q}$ on a Kodaira fibration $f\colon S\rightarrow B$.   Decompose the $\Q$-Hodge structure $V=H_1(F_b,\Q)$, for $F_b$ a very general fiber of $f$, as a sum of simple non-isomorphic $\Q$-Hodge structures 
\[V=\oplus_{i=1}^k V_i^{n_i}.\]
If there is some $V_i$ of dimension $2$, then the variation of $\Q$-Hodge structures $(p_i)_*(R_1f_*\mathbb{Q})$ induced by projection onto the $i$-th factor of $V$ is isotrivial.
\end{lem}

\begin{proof} Without loss of generality we may assume that $i=1$, so that $\dim V_1=2$.  If $L$ is the endomorphism algebra of $V$, let $D$ denote the connected component through $V$ of the Mumford-Tate domain of polarizable $\Q$-Hodge structures of type $\{(-1,0),(0,-1)\}$ with endomorphisms by $L$.  If $g$ is the genus of the fibers of the Kodaira fibration $f$, then the image $\overline{D}$ of $D$ in the moduli space $\mathcal{A}_g$ of principally polarized abelian varieties of dimension $g$ is of the form
\[\overline{D}\cong \mathcal{A}_1\times \overline{D}.'\]
Now, after a possible base change to get a locally constant level structure, the Kodaira fibration $f\colon S\rightarrow B$ corresponds to a complete curve $C$ in the fine moduli space $\mathcal{M}_{g}[n]$ of curves of genus $g$ with level $n\ge 3$ structure. Consider the image $\overline{C}$ of $C$ in $\mathcal{A}_g$. Letting $p_1$ denote the projection map of $D$ onto $\mathcal{A}_1$ via the above isomorphism, consider the image $(p_1)_*(\overline{C})$ of $\overline{C}$ in $\mathcal{A}_{1}$. Since the coarse moduli space $\mathcal{A}_{1}$ is isomorphic to the affine line via the map sending an elliptic curve to its $j$-invariant, we, in particular, have that $\mathcal{A}_{1}$ contains no complete curves. Hence the image of the projection $(p_1)_*(\overline{C})$ is just a finite union of points.  In particular, the variation of $\Q$-Hodge structures $(p_1)_*(R_1f_*\mathbb{Q})$ is isotrivial. 
\end{proof}

\begin{rem}\label{CMrem} In the above notation, if the simple $\Q$-Hodge structure $V_i$ in the decomposition $V=\oplus_{i=1}^k V_i^{n_i}$ is of CM type, meaning that the Hodge group of $V_i$ is commutative, then the Mumford-Tate domain $D_i$ through $V_i$ is a point and hence the variation of $\Q$-Hodge structures $(p_i)_*(R_1f_*\mathbb{Q})$ is isotrivial. 
\end{rem}

\begin{cor}\label{2dim} If $f\colon S\rightarrow B$ is a Kodaira fibration, then in the decomposition $V=\oplus_{i=1}^k V_i^{n_i}$, there must exist some simple Hodge structure $V_i$ of dimension at least $4$ on which the variation of $\Q$-Hodge structures  $(p_i)_*(R_1f_*\mathbb{Q})$ is not isotrivial. In particular, this $V_i$ is not of CM type. 
\end{cor}

\begin{proof}
By definition the fibration $f\colon S\rightarrow B$  is non-isotrivial and thus, as discussed in Section \ref{intro}, the variation of $\Q$-Hodge structures $R_1f_*\Q$ is not isotrivial. Hence the result follows from Lemma \ref{1dim} and Remark \ref{CMrem}. 
\end{proof}


\subsection{Weakly Special Varieties and Kodaira Fibrations}\label{mvz section}
Recall that a \emph{connected Shimura variety} $Z$ is the quotient of a Hermitian symmetric domain $X^+$ by a congruence subgroup $\Gamma$ of some reductive algebraic group $G$ defined over $\mathbb{Q}$ such that $G(\mathbb{R})$ acts on $X^+$ by conjugation. The pair $(X^+,G)$ is called the \emph{Shimura datum} for the Shimura variety $Z$.  For instance, the moduli space $\mathcal{A}_g[n]=\ker(Sp(2g,\Z)\rightarrow Sp(2g, \Z/n\Z))\backslash \mathcal{D}^3_g$ is a connected Shimura variety with Shimura datum $(\mathcal{D}^3_g, Sp(2g, \Q))$. 

Let $Z=\Gamma\backslash X^+$ be a connected Shimura variety with Shimura datum $(X^+, G)$. A \emph{weakly special variety} in $Z$ is either a Shimura subvariety of $Z$ or the image in $Z$ of $X_1^+\times \{x_2\}$ for some Shimura subdatum $(X_1\times X_2, G_1\times G_2)$ of $(X^+,G)$.

Moonen shows in \cite[Theorem 4.3]{moonenws} that an irreducible subvariety $Y$ of a Shimura variety $Z$ is weakly special if and only if $Y$ is totally geodesic. Moreover, Liu--Sun--Yang--Yau show in \cite[Theorem 1.4]{LSYY} that any compact submanifold in $\mathcal{M}_g$ whose image in $\mathcal{A}_g$ under the Torelli map is totally geodesic must be a ball quotient. Combining these, one deduces the following additional restriction on the summands of a Kodaira fibration.

\begin{lem}\label{shimura curve}
Let $f\colon S\rightarrow B$ be a Kodaira fibration whose very general fiber has rational homology $V=H_1(F_b,\Q)$ decomposing as a sum of simple $\Q$-Hodge structures 
\[V=V_1\oplus V_2\oplus \cdots \oplus V_n.\] 
If there is some $V_i$ of dimension $4$ with endomorphism algebra an indefinite quaternion algebra such that the variation of $\Q$-Hodge structures $(p_i)_*(R_1f_*\mathbb{Q})$ induced by projection onto the factor $V_i$ is not isotrivial, then there exists  $j\ne i$ such that $(p_j)_*(R_1f_*\mathbb{Q})$ is not isotrivial.
\end{lem}

\begin{proof}
Suppose there is some $4$-dimensional simple summand $V_i$ with endomorphism algebra $Q$ an indefinite quaternion algebra such that $(p_i)_*(R_1f_*\mathbb{Q})$ is not isotrivial.  As discussed in Section \ref{Hermitian symmetric}, the connected component through $V_i$ of the Mumford-Tate domain of $4$-dimensional $\Q$-Hodge structures of type $\{(-1,0),(0,-1)\}$ with endomorphisms by $Q$ is the Hermitian symmetric domain $\mathcal{D}^3_1$, which is just the upper half-plane $\mathbb{H}$. Its image inside $\mathcal{A}_2[n]$ is a compact Shimura curve $\mathcal{S}_Q$ given by taking the quotient of  $\mathbb{H}$ by the arithmetic group $\ker(SL(2,\mathcal{O}_Q)\rightarrow SL(2,\mathcal{O}_Q/n\mathcal{O}_Q))$ \cite[Proposition 9.2]{shimura2} (or see \cite[Section 5]{milne2}). 

Now suppose, for contradiction, that for every $j\ne i$, the variation of $\Q$-Hodge structures $(p_j)_*(R_1f_*\mathbb{Q})$ is isotrivial. Then we may base change to obtain a Kodaira fibration $f'\colon S'\rightarrow B'$ such that $(p_j)_*(R_1f'_*\mathbb{Q})$ is constant for every $j\ne i$. In other words, the only varying part of the cohomology of the fibers of $f'$ is given by the $4$-dimensional variation of $\Q$-Hodge structures $\mathbb{V}_i\coloneqq(p_i)_*(R_1f'_*\mathbb{Q})$, where by Lemma \ref{VHS lemma} for every $b\in B'$ there is an indefinite quaternion algebra $Q'$ embedding into the endomorphism algebra of $\mathbb{V}_{i,b'}$.

The base $B'$ of this Kodaira fibration thus yields a compact curve in  $\mathcal{M}_g$ whose image in $\mathcal{A}_g$ under the Torelli map is an \'etale cover of $\{E_0\}\times \mathcal{S}_Q$ for some elliptic curve $E_0$ and hence by  \cite[Theorem 4.3]{moonenws} is totally geodesic. But then by \cite[Theorem 1.4]{LSYY}, this image must be a complex ball quotient. But this contradicts the fact that $\mathcal{S}_Q$ is a quotient of the upper half plane $\mathbb{H}$. Hence $(p_i)_*(R_1f_*\mathbb{Q})$ must in fact be isotrivial. 
\end{proof}

\section{Genus $3$ Kodaira fibrations: the simple case}\label{section simple case}
We now specialize to the case of a Kodaira fibration $f\colon S\rightarrow B$ whose fibers have the smallest possible genus, meaning genus $3$, having the additional property that the Hodge structure $V=H_1(F_b,\Q)$ on a very general fiber $F_b$ is simple. We consider the possible endomorphism algebras of this simple $\Q$-Hodge structure $V$. Recall from Definition \ref{definition general complete intersection} the definition of a general complete intersection Kodaira fibration.  We thus, in particular, consider the possible endomorphism algebras of the simple Hodge structure $V$ when  $f\colon S\rightarrow B$ is a general complete intersection Kodaira fibration.

\begin{prop}\label{simple case}
The possible endomorphism algebras of a simple $\Q$-Hodge structure $V$ arising as $V=H_1(F_b,\Q)$, for $F_b$ a very general fiber of a genus $3$ Kodaira fibration, are:
\begin{enumerate}
\item \label{o1} $\Q$
\item\label {o2} A totally real field $L$ such that $[L:\Q]=3$.
\item \label{o3} An imaginary quadratic field.
\end{enumerate}
Of these possibilities, both \eqref{o1} and \eqref{o2} arise from general complete intersection Kodaira fibrations. By contrast, possibility \eqref{o3} cannot arise from a general complete intersection Kodaira fibration and thus is not known to actually occur. 
\end{prop}

\begin{proof}
Let $f\colon S\rightarrow B$ be a genus $3$ Kodaira surface such that the $\Q$-Hodge structure $V=H_1(F_b,\mathbb{Q})$ is simple. Let $L$ be the endomorphism algebra of $V$, let $F_0$ be the center of $L$, and let $F$ be the subfield of $F_0$ fixed by the Rosati involution (see Section \ref{background}). Writing $6=m[L:\mathbb{Q}]$ and $q^2=[L:F_0]$, observe that $q^2$ must divide $6$ and therefore we must have $q=1$. So, as discussed in Section \ref{background}, Albert's classification  \cite[Chapter X, \S11]{albert} yields that $L$ is either of Type I, namely $L$ is a totally real field, or $L$ is of Type IV and a CM field. Letting $l=[F:\mathbb{Q}]$, we know (see Section \ref{background}) that $l$ divides $3$. Hence if $L$ is of Type I, then either $L=\mathbb{Q}$ or $L$ is a totally real field of degree $3$ over $\mathbb{Q}$. If $L$ is of Type IV, then either $L$ is an imaginary quadratic field or $L$ is a CM field of degree $6$ over $\mathbb{Q}$. In the latter case, however, since $L$ is a CM field such that $\dim_LV=1$, the Hodge structure $V$ is of CM type \cite[Proposition V.3]{domain}. 

Since, by Remark \ref{CMrem}, the Hodge structure $V$ cannot be of CM type, we thus have only the following three possibilities for the endomorphism algebra $L$:
\begin{enumerate}
\item \label{p1}$L=\mathbb{Q}$
\item \label{p2}$L$ is a totally real field with $[L:\mathbb{Q}]=3$
\item \label{p3}$L$ is an imaginary quadratic field.
\end{enumerate}

Observe that by Remark \ref{complete intersection remark}, we know that possibility \eqref{p1} occurs and can by realized by a complete intersection Kodaira fibration. Now consider possiblity \eqref{p2}, namely the case when $L$ is a totally real field such that $[L:\mathbb{Q}]=3$. In this case, following Shimura \cite[Section 2.6]{shimura} (see Section \ref{background}), the connected component $D$ through $V$ of the Mumford-Tate domain of polarizable $\Q$-Hodge structures of type $\{(-1,0),(0,-1)\}$ with endomorphisms by $L$ is the product $\mathcal{D}^3_1\times \mathcal{D}^3_1\times \mathcal{D}^3_1$, where recall that $\mathcal{D}^3_1$ is just the upper half plane in $\mathbb{C}$. Consider the image $\overline{D}$ of $D$ inside the fine moduli space $\mathcal{A}_{3}[n]$, with $n\ge 3$. This Shimura variety $\overline{D}$ is the Hilbert modular threefold obtained by taking the quotient of $\mathcal{D}^3_1\times \mathcal{D}^3_1\times \mathcal{D}^3_1$ by $\ker(SL(2,\mathcal{O}_L)\rightarrow SL(2,\mathcal{O}_L/n\mathcal{O}_L))$. The boundary $\overline{D}^*\backslash \overline{D}$ of the Satake-Baily-Borel compactification $\overline{D}^*$ of $\overline{D}$ consists of a finite collection of points \cite{shimizu}. 

We now consider the elements of $\overline{D}\cap \mathcal{T}_3[n]^{\mathrm{dec}}$. Since for $g=3$ we have $\mathcal{T}_3[n]^{\mathrm{dec}}=\mathcal{A}_3[n]^{\mathrm{dec}}$, these are just the decomposable elements of $\overline{D}$. Suppose $w$ is a decomposable element of  $\overline{D}$, meaning that $w$ can be written as the product of two smaller-dimensional principally polarized abelian varieties. Then $w$ is the image under the quotient by $\ker(SL(2,\mathcal{O}_L)\rightarrow SL(2,\mathcal{O}_L/n\mathcal{O}_L))$ of some non-simple $6$-dimensional $\Q$-Hodge structure $W$ whose endomorphism algebra contains $L$. Consider the decomposition $W=\sum_{i=1}^kW_i^{n_i}$ of $W$ as a sum of simple non-isomorphic $\Q$-Hodge structures. Since $W$ has endomorphisms by $L$, each summand of the form $W_i^{n_i}$ must have endomorphisms by $L$.  If there were some $W_i$ of dimension $4$, then since $W$ has dimension $6$ we would have to have $n_i=1$, but then $3=[L:\mathbb{Q}]$ would need to divide $4$, which is impossible. Hence we must have $W= E^3$ for some $2$-dimensional $\Q$-Hodge structure $E$. From Section \ref{background}, the locus of $\Q$-Hodge structures $W$ of this form inside the Mumford-Tate domain $D$ is thus a countable union of copies of the upper half-plane $\mathcal{D}^3_1=\mathbb{H}$. Namely, the locus of decomposable principally polarized abelian varieties inside $\overline{D}$ has codimension $2$. Since we have established that the boundary $\overline{D}^*\backslash \overline{D}$ has codimension $3$, by Proposition \ref{complete intersection kod fib} there exists a general complete intersection Kodaira fibration realizing possibility \eqref{p2} on the list. 

The last possibility we consider is possibility \eqref{p3} on the list, namely the case when $L$ is an imaginary quadratic field.  Following Shimura \cite[Section 2.6]{shimura} (see Section \ref{background}), the corresponding Mumford-Tate domain $D$ through $V$ of $6$-dimensional polarizable $\Q$-Hodge structures of type $\{(-1,0),(0,-1)\}$ with endomorphisms by $L$ is the complex unit ball $\mathcal{D}^1_{2,1}=\{(u,v)\in \mathbb{C}^2\mid |u|^2+|v|^2<1\}.$ Consider the image $\overline{D}$ of $D$ inside the fine moduli space $\mathcal{A}_{3}[n]$, with $n\ge 3$.  This Shimura variety $\overline{D}$ is the Picard modular surface obtained by taking the quotient of $\mathcal{D}^1_{2,1}$ by $\ker(U((2,1)\mathcal{O}_L)\rightarrow U((2,1)\mathcal{O}_L/n\mathcal{O}_L))$ and its Satake-Baily-Borel compactification once again has boundary consisting of a finite collection of points \cite{holzapfel}.

As in the previous case, consider the elements of $\overline{D}\cap \mathcal{T}_3[n]^{\mathrm{dec}}$. These correspond to non-simple $\Q$-Hodge structures $W$ with endomorphisms by the imaginary quadratic field $L$. In this case however, the decomposition $W=\sum_{i=1}^kW_i^{n_i}$ into simple non-isomorphic $\Q$-Hodge structures can take multiple possible forms. One option is $W=E\oplus S$, where $E$ is a $2$-dimensional simple $\Q$-Hodge structure with endomorphisms by $L$ and $S$ is a $4$-dimensional $\Q$-Hodge structure with endomorphisms by $L$.  Another option is $W=E\oplus E'^2$, where $E$ is a $2$-dimensional $\Q$-Hodge structure with endomorphisms by $L$ and $E'$ is just a $2$ dimensional $\Q$-Hodge structure. The last option is $W=E'^3$ where $E'$ is a $2$-dimensional Hodge structure. From Section \ref{background}, the corresponding loci of $\Q$-Hodge structures $W$ of these forms inside the Mumford-Tate domain $D$ are all countable unions of the upper half-plane $\mathbb{H}$. Thus, in particular, the decomposable locus of $\overline{D}$ has codimension $1$ in $\overline{D}$. Namely, there is no general complete intersection Kodaira fibration realizing possibility \eqref{p3} on the list. 
\end{proof}

\begin{rem}\label{picard remark}
The proof of Proposition \ref{simple case} shows that if there were a genus $3$ Kodaira fibration such that the Hodge structure on a very general fiber had endomorphism algebra an imaginary quadratic field $L$, there would have to be a complete curve $C$ inside the Picard modular surface $\overline{D}$ parametrizing abelian varieties in $\mathcal{A}_3[n]$ with endomorphisms by $L$ such that $C$ avoided $\mathcal{A}_3[n]^{\mathrm{dec}}$. The proof also shows that $\overline{D}\cap \mathcal{A}_3[n]^{\mathrm{dec}}$ consists of a finite number of curves occurring as images of copies of the upper half plane inside the complex unit ball. 

Thus, while in general it is difficult to say whether there can be a complete curve $C$ on a Picard modular surface avoiding this finite set of curves, there are certain cases where the geometry of the Picard modular surface is understood well enough to say that this cannot occur. For instance, Holzapfel shows in \cite[Corollary 5.4.18]{holzapfel2} that if $L$ has discriminant $23$, then the minimal resolution of the Satake-Baily-Borel compactification of its corresponding Picard modular surface $\overline{D}$ is just the projective plane $\mathbb{P}^2$. Since any two  curves on $\mathbb{P}^2$ must intersect, any curve $C$ on $\overline{D}$ will intersect $\mathcal{A}_3[n]^{\mathrm{dec}}$. Namely, there can be no genus $3$ Kodaira fibration such that the Hodge structure on the general fiber has endomorphism algebra an imaginary quadratic field $L$ with discriminant $23$. 
\end{rem}


\section{Hecke Correspondences and Polarized Isogenies}\label{hecke section}
The case when the Hodge structure $V=H_1(F_b,\mathbb{Q})$ on a very general fiber is non-simple is much more subtle than the case when $V$ is simple. In particular, although $V$ decomposes as a $\mathbb{Q}$-Hodge structure, for $V$ to be the homology of a smooth curve the principally polarized complex abelian variety corresponding to $V$ cannot decompose. So if $V$ corresponds to a point $P\in \mathcal{A}_3[n]$, then $P$ is isogenous but not isomorphic to a point in $\mathcal{A}_3[n]^\mathrm{dec}$. Hence we are interested in the behavior of points of  $\mathcal{A}_3[n]^\mathrm{dec}$ under isogeny. 

In particular, in order to realize general complete intersection Kodaira fibrations with non-simple Hodge structure on the homology of their very general fibers, we will need to study the behavior of sub-loci of $\mathcal{A}_3[n]^\mathrm{dec}$ under compatible systems of isogenies on each point. We introduce the notion of Hecke correspondences below to make this idea more precise. 

Recall that a connected Shimura variety $Z$ is the quotient of a Hermitian symmetric domain $X^+$ by a congruence subgroup $\Gamma$ of some reductive algebraic group $G$ defined over $\mathbb{Q}$ such that $G(\mathbb{R})$ acts on $X^+$ by conjugation. 

\begin{defn}
Let $Z=\Gamma\backslash X^+$ be a connected Shimura variety. The \emph{Hecke correspondence} associated to an element $a\in G(\Q)$ consists of  the diagram
\[\begin{tikzcd}
\Gamma\backslash X^+&\arrow{l}[swap]{q}\Gamma_a\backslash X^+\arrow{r}{q_a}&\Gamma\backslash X^+
\end{tikzcd}\]
where 
\begin{itemize}
\item $\Gamma_a=\Gamma \cap a^{-1}\Gamma a$
\item $q(\Gamma_ax)=\Gamma x$
\item$ q_a(\Gamma_a x)=\Gamma ax$
\end{itemize}
In the above diagram, both $q$ and $q_a$ are finite morphisms of degree $[\Gamma:\Gamma_a]$. 

If $Z'$ is a closed irreducible subvariety of $Z$, any irreducible component of $q_a(q^{-1}Z)$ is called a \emph{Hecke translate} of $Z'$, which we denote by $aZ'$. 
\end{defn}

Suppose that $f\colon A\rightarrow B$ is an isogeny between principally polarized abelian varieties $(A,\lambda)$ and $(B,\mu)$. Then the pullback $f^*\mu$ is a polarization on $A$. We saw that $f$ is a \emph{polarized isogeny} if there exists an $n\in \mathbb{Z}$ such that $f^*\mu=n\lambda$. For a more thorough treatment of the distinction between isogenies and polarized isogenies see \cite{orr}. 

A point $Q\in \mathcal{A}_g[n]$ is a Hecke translate of another point $P\in \mathcal{A}_g[n]$ if and only if $P$ and $Q$ are related by a polarized isogeny, however the set of polarized isogenies between $P$ and $Q$ can be significantly smaller than the set of isogenies between them. For instance, in \cite[Proposition 3.1]{orr}, Orr shows that the isogeny class of an abelian surface with multiplication by a totally real quadratic field contains infinitely many distinct polarized isogeny classes. 

\begin{lem}\label{polarized isogeny lemma}
For $(A,\lambda)\in \mathcal{A}_g[n]$, consider the decomposition $V=\oplus_{i=1}^k V_i^{n_i}$ of $V=H_1(A,\mathbb{Q})$ into non-isomorphic simple $\mathbb{Q}$-Hodge structures. If each $n_i=1$ and $V_i$ has endomorphism algebra either $\mathbb{Q}$ or an imaginary quadratic field, then every isogeny class of $(A,\lambda)$ is a polarized isogeny class. 
\end{lem}
 
\begin{proof}
Let $L=\End(A)\otimes_\mathbb{Z} \mathbb{Q}$. So $L$ is the endomorphism algebra of $V$. Then under the given hypotheses $L$ decomposes as $L=\prod_{i=1}^kL_i$, where each $L_i$ is either $\mathbb{Q}$ or an imaginary quadratic field and $L$ is commutative. 

For any $(X,\nu)\in \mathcal{A}_g[n]$, let $\End^s X$ (respectively $\End^s_\mathbb{Q}X$) denote the subset of $\End X$ (respectively $\End_\mathbb{Q}X$) of endomorphisms of $X$ invariant under the Rosati involution on $\End_\mathbb{Q}X$ with respect to the polarization $\nu$.  Recall that the polarization $\nu$ is the class of an ample line bundle in $NS(A)$ and that $\nu$ induces an isogeny $\phi_\nu\colon A\rightarrow A^\vee$ of $A$ to its dual abelian variety $A^\vee$. Moreover, there is an isomorphism of groups $\delta\colon NS(X)\rightarrow \End^s X$ given by the restriction to $NS(X)$ of the map $\delta\colon NS_\mathbb{Q}(X)\rightarrow \End^s_\mathbb{Q}X$ defined by $\nu'\mapsto \phi_\nu^{-1}\phi_{\nu'}$ \cite[Proposition 5.2.1]{BL}. In particular, under the given hypotheses, we have $L^s=\prod_{i=1}^k\mathbb{Q}$ and an isomorphism of groups $NS(A)\cong \mathbb{Z}^k$. 

Suppose $(B,\mu)\in \mathcal{A}_g[n]$ is isogenous to $(A,\lambda)$ and let $f\colon A\rightarrow B$ be an isogeny. Then $f^*\mu$ is a polarization of $A$ and so we may view it as a tuple of integers $\prod_{i=1}^k n_i$. Precomposing the isogeny $f$ with the necessary elements of $\End V_i$, we may scale as needed to ensure that all of the integers $n_i$ are equal. This new isogeny $f'\colon A\rightarrow B$ is thus a polarized isogeny. 
\end{proof}


\section{Hecke Translation and the Decomposable Locus}\label{decomposable isogeny section}
Suppose that $f\colon S\rightarrow B$ is a genus $3$ Kodaira fibration such that the $\mathbb{Q}$-Hodge structure on a very general fiber $V=H_1(F_b,\Q)$ decomposes as $V=V_1\oplus V_2$.  Without loss of generality, we can assume that $\dim V_1=2$ and $\dim V_2=4$. Hence by Lemma \ref{1dim}, the variation of $\Q$-Hodge structures $(p_1)_*(R_1f_*\mathbb{Q})$ is isotrivial, meaning it is constant after possible base change. So $(p_1)_*(R_1f_*\mathbb{Q})$ corresponds to some fixed elliptic curve $E_0$. 

The question of  producing general complete intersection Kodaira fibrations with the above property then amounts to showing there exists a Hecke translate of the locus on points in $\mathcal{A}_3[n]^{\mathrm{dec}}$ of the form $E_0\times S$, for $S$ an abelian surface, such that this Hecke translate intersects the decomposable locus in codimension at least $2$. 

It is tempting to assume that, for a given sublocus $\mathcal{Z}$ of $\mathcal{A}_3[n]^{\mathrm{dec}}$ parametrizing abelian varieties with endomorphisms by some fixed ring, one can always find an element $a\in Sp(6,\mathbb{Q})$ such that $a\mathcal{Z}$ and $\mathcal{A}_3[n]^{\mathrm{dec}}$. 
In fact, suppose $\tilde{\mathcal{Z}}$ and $\tilde{\mathcal{A}}_3[n]^{\mathrm{dec}}$ are connected components of the lifts of $\mathcal{Z}$ and $\mathcal{A}_3[n]^{\mathrm{dec}}$ respectively to the Siegel upper half-space $\mathcal{D}^3_3$.
 Then Kleiman's results in \cite{kleiman} imply that there does exist an $a\in Sp(6,\mathbb{Q})$ such that $a\tilde{\mathcal{Z}}$ and $\tilde{\mathcal{A}}_3[n]^{\mathrm{dec}}$ meet transversally.
 However this does not imply that the images of $a\tilde{\mathcal{Z}}$ and $\tilde{\mathcal{A}}_3[n]^{\mathrm{dec}}$ in $\mathcal{A}_3[n]$ necessarily meet transversally.

 Indeed, suppose $\mathcal{Z}$ is the locus of points in $\mathcal{A}_3[n]^{\mathrm{dec}}$ of points of the form $E_0\times S$, for $E_0$ a fixed elliptic curve and $S$ an abelian surface such that $\End_\mathbb{Q}(S)$ contains a fixed indefinite quaternion algebra $L$.
  Then, as discussed in the proof of Lemma \ref{shimura curve}, the locus $\mathcal{Z}$ is a compact curve. Moreover, since $(p_1)_*(R_1f_*\mathbb{Q})$ is isotrivial, Lemma \ref{shimura curve} implies that every Hecke translate of $\mathcal{Z}$ has non-empty intersection with $\mathcal{A}_3[n]^{\mathrm{dec}}$.
 So in particular, no Hecke translate of $\mathcal{Z}$ meets $\mathcal{A}_3[n]^{\mathrm{dec}}$ transversally. 

Therefore, in order to prove the desired transversality results for the particular locus in $\mathcal{A}_3[n]^{\mathrm{dec}}$ that we want to consider, it will be necessary to first investigate in detail how Hecke translation affects the decomposability of points in various subloci of $\mathcal{A}_3[n]^{\mathrm{dec}}$.

Hence, fix an elliptic curve $E_0$ and consider the subsets of $\mathcal{A}_3[n]^{\mathrm{dec}}$ given by
\begin{equation}\label{bce}
\begin{aligned}
\mathcal{B}&=\{E_0\times S\in \mathcal{A}_3[n]^{\mathrm{dec}}\mid S \text{ is an abelian surface } \}\\
\mathcal{C}&=\{E_0\times S\in \mathcal{A}_3[n]^{\mathrm{dec}}\mid S\sim E_1\times E_2 \text{ for some elliptic curves } E_1, E_2\}\\
\mathcal{E}&=\{E_0\times S\in \mathcal{A}_3[n]^{\mathrm{dec}}\mid S\sim E^2 \text{ for some elliptic curve } E\}.
\end{aligned}
\end{equation}
Here both $\mathcal{C}$ and $\mathcal{E}$ decompose as countable unions $\mathcal{C}=\cup_i\mathcal{C}_i,$ and $\mathcal{E}=\cup_i\mathcal{E}_i,$ where each $\mathcal{C}_i=\{E_0\}\times \mathcal{S}_i$ for $\mathcal{S}_i$ a connected Shimura surface in $A_2[n]$ parametrizing abelian surfaces isogenous to a product of elliptic curves and each $\mathcal{E}_i=\{E_0\}\times \mathcal{T}_i$ for $\mathcal{T}_i$ a connected Shimura curve in $A_2[n]$ parametrizing abelian surfaces isogenous to a square of an elliptic curve.

Recall that Poincar\'{e}'s reducibility theorem states that any abelian variety $A$ has a decomposition (up to isogeny) $A\sim \prod_{i=1}^k A_i^{n_i}$, where none of the $A_i$ are isogenous \cite[Theorem 5.3.5]{BL}. Note that although the $A_i$ are not unique in this decomposition because they are only uniquely determined after tensoring with $\Q$, the values $k$, $\dim A_i$, and $n_i$ are are uniquely determined. 
These $A_i$ are the \emph{simple factors} of $A$. Note that using the equivalence between complex abelian varieties up to isogeny and polarizable $\mathbb{Q}$-Hodge structures of type $\{(-1,0),(0,-1)\}$, an abelian variety is simple if and only if its associated $\mathbb{Q}$-Hodge structure $A_\mathbb{Q}$ is simple. 

\begin{lem}\label{simple lemma}
Let $P=E_0\times S$ be an element of $\mathcal{B}$ such that $S$ is simple. If $a\in Sp(6,\mathbb{Q})$ is such that $aP\in \mathcal{A}_3^{\mathrm{dec}}[n]$, then $aQ\in  \mathcal{A}_3^{\mathrm{dec}}[n]$ for any point $Q$ of the form $E'\times S$, for $E'$ an elliptic curve, or of the form $E_0\times T $, for $T$ an abelian surface which does not have the elliptic curve $E_0$ as a simple factor. 
\end{lem}
\begin{proof}
Since $aP$ lies in $\mathcal{A}_3^{\mathrm{dec}}[n]$, we may write $aP=E'\times S'$ for some elliptic curve $E'$ and some abelian surface $S'$. Moreover, we know that $a$ induces an isomorphism of $\mathbb{Q}$-Hodge structures $(E_0)_\mathbb{Q}\times S_\mathbb{Q}\cong (E')_\mathbb{Q}\times S'_\mathbb{Q}$. Since the category of polarizable $\mathbb{Q}$-Hodge structures is semisimple and, by assumption, we know that $S_\mathbb{Q}$ is simple, it follows that $a$ induces isomorphisms $(E_0)_\mathbb{Q}\cong (E')_\mathbb{Q}$ and $S_\mathbb{Q}\cong S'_\mathbb{Q}$. In particular, the isogeny of $P$ induced by $a$ decomposes as an isogeny of $E_0$ and an isogeny of $S$. 

Hence if $Q=E'\times S$ for some elliptic curve $E'$ or if if $Q=E_0\times T$ for an abelian surface $T$ which does not have $E_0$ as a simple factor, then since there are no homomorphisms between non-isogenous abelian varieties, the isogeny of $Q$ induced by $a$ also decomposes as a product of isogenies, one on the elliptic curve and one on the abelian surface. It follows that $aQ$ is also decomposable. 
\end{proof}

\begin{lem}\label{elliptic curve lemma}
Let $P=E_0\times S$ be an element of $\mathcal{C}_i$ such that $S\sim E_1\times E_2$ with no two of the $E_j$ being isogenous. If $a\in Sp(6,\mathbb{Q})$ is such that $aP\in \mathcal{A}_3^{\mathrm{dec}}[n]$, then $aQ\in  \mathcal{A}_3^{\mathrm{dec}}[n]$ for every other point $Q=E_0\times T \in \mathcal{C}_i$ such that $T\sim E_1'\times E_2'$ with no two of $E_0, E_1', E_2'$ being isogenous. 
\end{lem}

\begin{proof}
Let $Q\in  \mathcal{C}_i$ be a point satisfying the hypotheses of the Lemma. Since $aP$ lies in $\mathcal{A}_3^{\mathrm{dec}}[n]$, we may write $aP=E'\times S'$ for some elliptic curve $E'$ and some abelian surface $S'$. Moreover, we know that $a$ induces an isomorphism of $\mathbb{Q}$-Hodge structures $(E_0)_\mathbb{Q}\times (E_1)_\mathbb{Q}\times (E_2)_\mathbb{Q}\cong (E')_\mathbb{Q}\times S'_\mathbb{Q}$. Since the category of polarizable $\mathbb{Q}$-Hodge structures is semisimple, it follows that $a$ induces an isomorphism $(E')_\mathbb{Q}\cong E_j$ for some $j\in \{0,1,2\}$. If $(E')_\mathbb{Q}\cong E_0$, then, because no two of the $E_j$ are isogenous and thus there are no homomorphisms among the factors of $P$ and no homomorphisms amongst the factors of $Q$, the argument to show that $aQ$ is decomposable proceeds identically to that of the proof of Lemma \ref{simple lemma}. 

So without loss of generality suppose that $a$ instead induces an isomorphism $(E')_\mathbb{Q}\cong (E_2)_\mathbb{Q}$. Hence $E'\sim E_2$ and $S'\sim E_0\times E_1$. Let $a_P$ and $a_Q$ denote the isogenies of $P$ and $Q$ respectively induced by $a$. Now since $S\sim E_1\times E_2$ and moreover none of the $E_j$ are isogenous, we may factor the isogeny $a_P$ as a composition 
\begin{equation}\label{isogeny factor} E_0\times S\rightarrow E_0\times E_1\times E_2\rightarrow S'\times E'.\end{equation}

Now for any point $R$ of one of the Shimura surfaces $\mathcal{S}_j$, let $\tilde {R}$ denote a preimage of $R$ in the Hermitian symmetric domain $\mathcal{D}_2^3$ of $4$-dimensional $\mathbb{Q}$-Hodge structures of type $\{(-1,0),(0-1)\}$. Then the Shimura surface $\mathcal{S}_j$ lifts to the Hermitian symmetric domain $\mathcal{D}_1^3\times \mathcal{D}_1^3\subset \mathcal{D}_2^3$ passing through $\tilde {R}$. Moreover, every point of this $\mathcal{D}_1^3\times \mathcal{D}_1^3$ lies in the orbit of $R$ under the action of $SL(2,\mathbb{R})\times SL(2,\mathbb{R})$. 

For an element $\alpha \in SL(2,\mathbb{R})\times SL(2,\mathbb{R})$, denote by $\alpha R$  the image in $\mathcal{A}_3[n]$ of $\alpha\tilde{R}$. Thus we may write $T=\alpha S$ for some $\alpha\in SL(2,\mathbb{R})\times SL(2,\mathbb{R})$. Let us write $\alpha=(\alpha_1,\alpha_2)$ with respect to the two factors $SL(2,\mathbb{R})\times SL(2,\mathbb{R})$. Then the isogeny $E_0\times S\rightarrow E_0\times E_1\times E_2$ induces an isogeny 
\[E_0\times \alpha S \rightarrow E_0\times \alpha_1E_1\times \alpha_2E_2.\] 
So then using Equation \eqref{isogeny factor}, the isogeny $a_Q$ of $Q$ induced by the element $a\in Sp(6,\mathbb{Q})$ factors as 
\begin{equation}\label{isogeny factor 2}
E_0\times T\rightarrow E_0\times \alpha(E_1\times E_2)\rightarrow T'\times E''\end{equation}
for some abelian surface $T'\sim E_0\times \alpha_1E_1$ and some elliptic curve $E''\sim \alpha_2E_2.$ In particular, we have that $aQ$ is decomposable. 
\end{proof}


\section{The non-simple case}\label{non-simple section}
We are now ready to address the case of possible endomorphism algebras of the $\mathbb{Q}$-Hodge structure on a very general fiber of a genus $3$ Kodaira fibration. The strategy to show that there exist general complete intersection Kodaira fibrations with decomposable Hodge structure on the general fiber is to produce an element $a\in Sp(6,\Q)$ such that the Hecke translate $a\mathcal{B}$ of the set $\mathcal{B}$ defined in \eqref{bce} intersects the decomposable locus $\mathcal{A}_3[n]^{\mathrm{dec}}$ in codimension at least $2$. 

We will accomplish this by defining a proper subset $\Gamma\Gamma_0\subset Sp(6,\Q)$ such that every $a\in Sp(6,\Q)$ in the complement of $\Gamma\Gamma_0$ has the above property.

So, fix a point $P_0=E_0\times E_1\times E_2$ in $\mathcal{B}$ such that none of the $E_i$ are isogenous to each other. Note that for any elliptic curve $E$ the endomorphism algebra of $H_1(E,\Q)$ is either $\Q$ or an imaginary quadratic field. So because none of the $E_i$ are isogenous to each other, it follows from Lemma \ref{polarized isogeny lemma} that if $P\sim Q$ for some point $Q\in \mathcal{A}_3[n]$, then there exists an element $a\in Sp(6,\Q)$ such that $aP=Q$. 

Consider the following subset of $Sp(6,\Q)$:
\[\Gamma=\{a\in Sp(6,\mathbb{Q})\mid aP_0 \in \mathcal{A}_3^{\mathrm{dec}}[n]\}.\]
We now define the following three subsets of $\Gamma$
\begin{equation*}
\begin{aligned}
\Gamma_0&=\{a\in \Gamma \mid aP_0=E_0'\times S', \text{ where } E_0'\sim E_0 \text{ and } S'\sim E_1\times E_2\}\\
\Gamma_1&=\{a\in \Gamma \mid aP_0=E_1'\times S', \text{ where } E_1'\sim E_1 \text{ and } S'\sim E_0\times E_2\}\\
\Gamma_2&=\{a\in \Gamma \mid aP_0=E_2'\times S', \text{ where } E_2'\sim E_2 \text{ and } S'\sim E_0\times E_1\}.
\end{aligned}
\end{equation*}
Since by definition the elements of the sets $\Gamma_i$ all lie in $\Gamma$ and thus in particular they send $P_0$ into the decomposable locus, it follows that we have the decomposition 
\[\Gamma=\Gamma_0\cup \Gamma_1\cup \Gamma_2.\]
Additionally, note that for any $i\ne j$ we have
\begin{equation}\label{gamma intersection equation} \Gamma_i\cap \Gamma_j=\{a\in \Gamma \mid aP_0=E_0'\times E_1'\times E_2', \text{ where } E_0'\sim E_0, E_1'\sim E_1, \text{ and } E_2'\sim E_2\}.\end{equation}

\begin{lem}\label{Gamma_i lemma}
For $i\in \{0,1,2\}$, let $Q=E\times S$ be a point such that $E\sim E_i$ and no simple factor of $S$ is isogenous to $E_i$.  Let $a\in Sp(6,\Q)$. Then $a\in \Gamma_i$ if and only if $aQ=E'\times S'$, where $E'\sim E$ and $S'\sim S$. 
\end{lem}

\begin{proof}
Suppose  $a\in \Gamma_i$. Then we know $aP_0=E_i''\times S''$, where $E_i''\sim E_i$ and $S''\sim E_j\times E_k$, where $i,j,k$ are distinct. 
In particular, the isogeny of $P_0$ induced by $a$ decomposes as the product of an isogeny of $E_i$ times an isogeny of the surface $E_j\times E_k$. Since there are no homomorphisms between non-isogenous abelian varieties and the surface $S$ has no simple factors isogenous to $E_i$, it follows that the isogeny of $Q$  induced by $a$ also decomposes as the product of an isogeny of $E$ and an isogeny of $S$. Hence $aQ$ is of the form $E'\times S'$ where $E'\sim E$ and $S'\sim S$.

Conversely, suppose that $aQ$ is of the form $E'\times S'$ where $E'\sim E$ and $S'\sim S$. Then by the same argument as above, the isogeny of $Q$ induced by $a$ decomposes as the product of an isogeny of $E$ and an isogeny of $S$. Hence since $P_0=E_0\times E_1\times E_2$, where none of the elliptic curve factors are isogenous, it follows that the isogeny of $P_0$ induced by $a$ decomposes as an isogeny of $E_i$ times an isogeny of $E_j\times E_k$. Thus $a\in \Gamma_i$. 
\end{proof}

\begin{cor}\label{group corollary}
Each of the subsets $\Gamma_i\subset \Gamma$ is an algebraic subgroup of $Sp(6,\Q)$ isomorphic to $SL(2,\Q)\times Sp(4,\Q)$. Moreover for $i\ne j$ the subset $\Gamma_i\cap \Gamma_j\subset \Gamma$ is an algebraic subgroup of $Sp(6,\Q)$ isomorphic to $SL(2,\Q)^3$.
\end{cor}

\begin{proof}
Note that each $\Gamma_i$ is the preimage of an irreducible component of $\mathcal{A}_3^{\mathrm{dec}}[n]$ under the map $Sp(6,\Q)\rightarrow \mathcal{A}_3[n]$ given by $a\mapsto aP_0$. Hence $\Gamma_i$ is a Zariski-closed subset of $Sp(6,\Q)$. We now check that $\Gamma_i$ is a subgroup of $Sp(6,\Q)$. 

It is clear that $\Gamma_i$ contains $\mathrm{Id}$. Suppose $a\in \Gamma_i$ and consider the point $Q=aP_0$. Then $Q=E_i'\times S'$ where $E_i'\sim E_i$ and $S'\sim E_j\times E_k$ with $i,j,k$ distinct. Since $a^{-1}Q=P_0$, Lemma \ref{Gamma_i lemma} implies $a^{-1}\in \Gamma_i$. 

Similarly, suppose $a,b\in \Gamma_i$ and consider the element $ab\in Sp(6,\Q)$. Since $b\in \Gamma_i$ we know $bP_0=E_i'\times S'$, where $E_i'\sim E_i$ and $S'\sim E_j\times E_k$, with $i, j , k$ distinct. Then by Lemma \ref{Gamma_i lemma}, since $a\in \Gamma_i$, we have $a(bP_0)$ of the form $E_i''\times S''$, where $E_i''\sim E_i$ and $S''\sim S'$. Since $S'\sim E_j\times E_k$, it follows that $ab\in \Gamma_i$. Thus $\Gamma_i$ is indeed an algebraic subgroup of $Sp(6,\Q)$. 

It is clear that the subgroup $SL(2,\Q)\times Sp(4,\Q)$ of $Sp(6,\Q)$ embeds into the set $\Gamma_i$. Moreover, suppose $a \in \Gamma_i$. Then the isogeny of $P_0$ induced by $a$ can be written as the product of an isogeny of $E_i$ times an isogeny of 
$E_j\times E_k$ with $i,j,k$ distinct. Since none of the elliptic curves $E_0$, $E_1$, $E_2$ are isogenous, if follows from Lemma \ref{polarized isogeny lemma} that there is a corresponding element of $a_1\in SL(2,\Q)$ and an element $a_2\in Sp(4,\Q)$ such that $aP_0=a_1E_i\times a_2(E_j\times E_k)$. Hence $SL(2,\Q)\times Sp(4,\Q)$ surjects onto $\Gamma_i$. This establishes the isomorphism $\Gamma_i\cong SL(2,\Q)\times Sp(4,\Q)$

For $i\ne j$ the subset $\Gamma_i\cap \Gamma_j$ is the intersection of two Zariski closed subgroups of $Sp(6,\Q)$, so it too is an algebraic group. Moreover, it comes equipped with a natural embedding $SL(2,\Q)^3\hookrightarrow \Gamma_i\cap \Gamma_j$. As in the case of the group $\Gamma_i$, the characterization \eqref{gamma intersection equation} together with Lemma \ref{polarized isogeny lemma} yields that this embedding is a surjection, whence the isomorphism $\Gamma_i\cap \Gamma_j\cong SL(2,\Q)^3$. 
\end{proof}

Now consider the product $\Gamma\Gamma_0 \subset Sp(6,\Q)$. Note that since $\Gamma=\Gamma_0\cup \Gamma_1\cup \Gamma_2$, we have the decomposition
\[\Gamma\Gamma_0=(\Gamma_1\Gamma_0)\cup (\Gamma_2\Gamma_0).\]
It follows from Corollary \ref{group corollary} that both $\Gamma_1\Gamma_0$ and $\Gamma_2\Gamma_0$ have dimension $(3+10) + (3+10)-3(3)=17$. Since $Sp(6,\Q)$ is $21$-dimensional, it follows in particular that  $\Gamma\Gamma_0$ is a proper subset of $Sp(6,\Q)$.

\begin{lem}\label{GammaGamma_0 lemma}
Let $Q\in \mathcal{B}$ be such that $Q=E_0\times S$ with either $S$ simple or $S\sim E_1'\times E_2'$ such that no two of the elliptic curves $E_0$, $E_1'$, or $E_2'$ are isogenous. If $a\in Sp(6,\Q)$ is such that $aQ\in \mathcal{A}_3^{\mathrm{dec}}[n]$, then $a\in \Gamma\Gamma_0$. 
\end{lem}

\begin{proof}
If the abelian surface $S$ is simple, then by Lemma \ref{simple lemma}, we know $aP_0\in \mathcal{A}_3^{\mathrm{dec}}[n]$, hence $a\in \Gamma$. So in particular $a\in \Gamma\Gamma_0$. 

Now suppose that  $S\sim E_1'\times E_2'$ such that no two of the elliptic curves $E_0$, $E_1'$, or $E_2'$ are isogenous. The point $Q=E_0\times S$ must lie on one of the connected surfaces $\mathcal{C}_i\subset \mathcal{C}$. Moreover, there exists a point $R=E_0\times T$ on this surface $\mathcal{C}_i$ such that $T\sim E_1\times E_2$. By Lemma \ref{elliptic curve lemma}, we know $aR\in \mathcal{A}_3^{\mathrm{dec}}[n]$. 

Since $T\sim E_1\times E_2$, there exists an isogeny from the point $P_0$ to the point $R$. Moreover, since none of the $E_i$ are mutually isogenous, it follows from Lemma \ref{polarized isogeny lemma} that in fact there exists an element $\gamma\in Sp(6,\Q)$ such that  $\gamma P_0=R$. By construction, we know $\gamma\in \Gamma_0$. 

Since $a\gamma(P_0)\in \mathcal{A}_3^{\mathrm{dec}}[n]$, it follows that $a\gamma \in \Gamma$. We know by Lemma \ref{group corollary} that $\Gamma_0$ is a group, so in particular $\gamma^{-1}\in \Gamma_0$. It follows that $a\in \Gamma\Gamma_0$. 
\end{proof}

\begin{cor}\label{complement corollary}
If $a\in Sp(6,\Q)$ is in the complement of $\Gamma\Gamma_0$, then $a\mathcal{B}\cap \mathcal{A}_3^{\mathrm{dec}}[n]$ has codimension at least $2$ in $a\mathcal{B}$.
\end{cor}

\begin{proof}
By Lemma \ref{GammaGamma_0 lemma}, if $a\in Sp(6,\Q)$ is in the complement of $\Gamma\Gamma_0$, then the only points of $\mathcal{B}$ that can be sent into $\mathcal{A}_3^{\mathrm{dec}}[n]$ by $a$ are those of the form $Q=E_0\times S$, where either $S\sim E_0\times E$ for some elliptic curve $E$ or $S\sim E^2$. The locus of points of this form is thus $1$-dimensional inside the $3$-dimensional locus $\mathcal{B}$. 
\end{proof}

\begin{rem}
It is possible to define an analogous subset to $\Gamma\Gamma_0$ for the subset $\mathcal{C}$ defined in \eqref{bce}. However Lemma \ref{polarized isogeny lemma} does not apply to the points of $\mathcal{E}$ (also defined in \eqref{bce}), which forms a codimension $1$ subset of $\mathcal{C}$. In the absence of Lemma \ref{polarized isogeny lemma}, it is difficult to compute the dimension of this analogous subset in order to show it is properly contained in $Sp(6,\Q)$. Additionally, even if one could show proper containment, without Lemma \ref{polarized isogeny lemma}, it is difficult to show that an element in the complement sends all but countably many of the points in $\mathcal{E}$ out of the decomposable locus. 
\end{rem}

\begin{prop}\label{prop non-simple case} The possible endomorphism algebras of a $4$-dimensional $\Q$-Hodge structure $V_2$ arising as $H_1(F_b,\Q)=V_1\oplus V_2$, for $F_b$ a very general fiber of a genus $3$ Kodaira fibration, are:
\begin{enumerate}
\item\label{m1}  $\Q$
\item \label{m2}A totally real quadratic field
\end{enumerate}
Possibility \eqref{m1} can be realized from a general complete intersection Kodaira fibration, while \eqref{m2} is not known to occur.
\end{prop}

\begin{proof}
Let $f\colon S\rightarrow B$ be a genus $3$ Kodaira fibration such that for a very general fiber $F_b$ the $\Q$-Hodge structure $V=H_1(F_b,\Q)$ decomposes as $V=V_1\oplus V_2$, where $V_i$ has dimension $2i$ for $i=1,2$. Recall from Lemma \ref{1dim} that because $V_1$ is $2$-dimensional, the variation of $\Q$-Hodge structures $(p_1)_*(R_1f_*\mathbb{Q})$ induced by projection onto the $V_1$ factor of $H_1(F_b,\mathbb{Q})$ must be isotrivial. After a base change of the curve $B$, we can assume that the variation  $(p_1)_*(R_1f_*\mathbb{Q})$ is in fact constant. Moreover, by Corollary \ref{2dim}, the variation of $\Q$-Hodge structures $(p_2)_*(R_1f_*\mathbb{Q})$ induced by projection onto the $V_2$ factor of $H_1(F_b,\mathbb{Q})$ must be non-trivial. Hence $V_2$ is not of CM type and by Lemma \ref{1dim} must be simple as a Hodge structure.

Let $L$ be the endomorphism algebra of $V_2$ and let $F_0$ be its center. Recall from Section \ref{background} Albert's classification of possible endomorphism algebras of a $4$-dimensional simple polarizable $\Q$-Hodge structure of type $\{(-1,0),(0,-1)\}$. Write $4=m[L:\Q]$ and $q^2=[L:F_0]$. Hence either $q=1$ or $q=2$. If $q=1$, then $L$ is commutative, so $L$ if of Type I or IV in Albert's classification. If $L$ is of Type I, then, as discussed in Section \ref{background}, we have that $L$ is a totally real field and $[L:\Q]$ divides $2$, so $L=\Q$ or $L$ is a totally real quadratic field. If $L$ is of Type IV and $q=1$, then $L$ is a CM field of degree $2$ or $4$ over $\Q$. However, Shimura shows in \cite[Theorem 5]{shimura} that a $4$-dimensional simple polarizable $\Q$-Hodge structure of type $\{(-1,0),(0,-1)\}$ cannot have endomorphism algebra an imaginary quadratic field, which eliminates the first possibility. Moreover, since $V_2$ cannot be of CM type, the case of a CM field of degree $4$ over $\Q$ is also not possible. If $q=2$, then Albert's classification yields that $L$ is a quaternion algebra of either Type II or III. However, Shimura shows in \cite[Theorem 5]{shimura} that the Type III case also cannot occur. Therefore we are left with the following three possibilities for the endomorphism algebra $L$:
\begin{enumerate}
\item\label{n1} $\Q$
\item\label{n2} A totally real quadratic field
\item\label{n3} An indefinite quaternion algebra
\end{enumerate}

It follows from Lemma \ref{shimura curve} that possibility \eqref{n3} cannot occur as then the variation of $\Q$-Hodge structures $R_1f_*\mathbb{Q}$ would be isotrivial, so we eliminate this possibility from the list.

In order to show that possibility \eqref{n1} arises from a general complete intersection Kodaira fibration, fix an elliptic curve $E_0$ and consider the subset $\mathcal{B}$ defined in \eqref{bce}. Recall the proper subset $\Gamma\Gamma_0\subset Sp(6,\Q)$ defined in the earlier part of this section. Let $a\in Sp(6,\Q)$ lie in the complement of $\Gamma\Gamma_0$ in $Sp(6,\Q)$. Then by Corollary \ref{complement corollary}, the Hecke translate $a\mathcal{B}\cap \mathcal{A}_3^{\mathrm{dec}}[n]$ has codimension at least $2$ in $a\mathcal{B}$.

Moreover, we have that $\mathcal{B}$ is contained in the connected Shimura variety $\mathcal{A}_1[n]\times \mathcal{A}_2[n]$ and so $a\mathcal{B}$ is contained in $a(\mathcal{A}_1[n]\times \mathcal{A}_2[n])$. The Satake-Baily-Borel compactification $\mathcal{A}_1^*[n]\times \mathcal{A}_2[n]^*$, then induces compactifications of $\mathcal{B}$ and hence of $a\mathcal{B}$ so that this compactification $(a\mathcal{B})^*$ is isomorphic to $\mathcal{A}_2[n]^*$. Since $\mathcal{A}_2[n]^*\backslash \mathcal{A}_2[n]$ has codimension $2$ in $\mathcal{A}_2[n]^*$, the boundary $(a\mathcal{B})^*\backslash (a\mathcal{B})$ has codimension $2$ in $(a\mathcal{B})^*$.

Thus, recalling that $\mathcal{A}_3[n]=\mathcal{T}_3[n]$, we have shown that $a\mathcal{B}$ verifies the conditions of Proposition \ref{complete intersection kod fib}. So there indeed exists a general complete intersection curve corresponding to a general complete intersection Kodaira fibration with the property that the homology of a very general fiber $F_b$ decomposes as $H_1(F_b,\Q)=V_1\oplus V_2$, where $\End_\mathbb{Q}(V_2)=\mathbb{Q}$. 
 \end{proof}

\section{Connected Monodromy of Genus $3$ Kodaira Fibrations}\label{section main result}
We now use the results of the preceding sections to characterize the connected monodromy groups of genus $3$ Kodaira fibrations. Before stating the theorem, we introduce a bit of notation. 

For a field extension $K'/K$ and an algebraic group $G$ defined over $K'$ let $R_{K'/K}G$ denote the Weil restriction functor from $K'$ to $K$ applied to the group $G$. Moreover, if $V$ is a $K$-vector space such that $K'\subset \End_K(V)$, let $_{K'}V$ denote $V$ considered as an $K'$-vector space via its action by endomorphisms. 

Lastly, one of the cases of the theorem involves the situation that the endomorphism algebra $L$ of $V=H_1(F_b,\Q)$, for $F_b$ a very general fiber of a Kodaira fibration, is an imaginary quadratic field. If $M\cong M_3(L^{\mathrm{op}})$ is the centralizer of $L$ in $\End_\Q(V)$, let $^{-}$ denote the involution on $M$ induced by complex conjugation on $L$. Then define
\begin{align*}
U(M,^{-})&\coloneqq\{m\in M^*\mid \overline{m}=m^{-1}\}\\
SU(M,^{-})&\coloneqq \ker \left(\mathrm{Norm}_{M/\Q}\colon U(M,^{-})\rightarrow \mathbb{G}_{m,\Q}\right).
\end{align*}
The groups $U(M,^{-})$ and $SU(M,^{-})$ are isomorphic over $\overline{\Q}$ to the groups $U(3)$ and $SU(3)$ respectively (see for instance \cite{boi}).

Moreover, for the purpose of translating between this Theorem \ref{main theorem} and Theorem \ref{complex theorem} stated in the introduction, note that if $L$ is a totally real field with $[L:\mathbb{Q}]=3$, then over $\mathbb{C}$ the group $R_{L/\mathbb{Q}}SL(_LV)$ is isomorphic to $SL(2)^3$. Similarly, if $V_2$ is a $4$-dimensional $\mathbb{Q}$-vector space such that $L_2\subset \End_{\mathbb{Q}}V_2$, then over $\mathbb{C}$ the group $R_{L_2/\mathbb{Q}}SL(_{L_2}V_2)$ is isomorphic to $SL(2)^2$.

\begin{thm}\label{main theorem}
Suppose $f\colon S\rightarrow B$ is a genus $3$ Kodaira fibration with very general fiber $F_b$ and associated $\Q$-Hodge structure $V=H_1(F_b,\Q)$ having endomorphism algebra $L$. Then the only possible connected monodromy groups realized by such a fibration are described by the following:
\begin{enumerate}
\item If $V$ is simple:
\begin{enumerate}
\item \label {1a}The group $Sp(6)$ if $L=\Q$
\item \label{1b} The group $R_{L/\mathbb{Q}}SL(_LV)$ if $L$ is a totally real field with $[L:\mathbb{Q}]=3$
\item \label{1c} The group $SU(M,^{-})$, if $L$ is an imaginary quadratic field, where $M\cong M_2(L^{\mathrm{op}})$ is the centralizer of $L$ in $\End_\Q(V)$
\end{enumerate}
\item If $V$ has $\Q$-Hodge structure decomposition $V=V_1\oplus V_2$, where $\dim_\Q V_i=2i$ and $V_2$ has endomorphism algebra $L_2$, then the connected monodromy group acts trivially on $V_1$ and its action on $V_2$ is given by:
\begin{enumerate}
\item \label{2a} The group $Sp(4)$ if $L_2=\mathbb{Q}$
\item \label{2b}The group $R_{L_2/\mathbb{Q}}SL(_{L_2}V_2)$ if $L_2$ is a totally real quadratic field
\end{enumerate}
\end{enumerate}
Moreover, Cases \eqref{1a}, \eqref{1b}, and \eqref{2a} are all realized by general complete intersection Kodaira fibrations, while Cases  \eqref{1c} and \eqref{2b} are not known to occur.
\end{thm}

\begin{proof}
Suppose $f\colon S\rightarrow B$ is a Kodaira fibration with $V=H_1(F_b,\Q)$ for $F_b$ a very general fiber. Then it follows from Theorem \ref{andrethm}\cite[Theorem 1]{andre} that the connected monodromy group $T$ of the fibration $f$ is a normal subgroup of $[Hg(V),Hg(V)]$, where $Hg(V)$ denotes the Hodge group of $V$.
Moreover, for $p$ a prime, Ribet proves in \cite{ribet1} that the Hodge group of a $p$-dimensional simple abelian variety $A$ is completely determined by $\End_\Q(A)$ and is given by the connected component of the centralizer of $\End_\Q(A)$ in $Sp(W)$, where $W=H_1(A,\Q)$. Using the equivalence of categories discussed in Section \ref{background} between complex abelian varieties up to isogeny and polarizable $\Q$-Hodge structures of type $\{(-1,0), (0,-1)\}$ one obtains the same result for any simple $2p$-dimensional $\Q$-Hodge structure $W$ of type $\{(-1,0), (0,-1)\}$. Namely that $Hg(W)$ is equal to the connected component of the centralizer of the endomorphism algebra of $W$ in $Sp(W)$. 

Thus, if the Hodge structure $V$ is simple, then it follows from Proposition \ref{simple case} that the possibilities for the endomorphism algebra $L$ of $V$ are exactly those described in Case \eqref{1a}, \eqref{1b}, and  \eqref{1c}. 
By Ribet's result the corresponding Hodge groups are the group $Sp(6)$ if $L=\Q$, the group $R_{L/\mathbb{Q}}SL(_LV)$ if $L$ is a totally real field with $[L:\mathbb{Q}]=3$, and the group $U(M,^{-})$ if $L$ is an imaginary quadratic field, where $M\cong M_2(L^{\mathrm{op}})$ is the centralizer of $L$ in $\End_\Q V$. The corresponding derived subgroups $Sp(6)$, $R_{L/\mathbb{Q}}SL(_LV)$, and $SU(M,^{-})$ are all normal, so the result follows from Theorem \ref{andrethm} \cite[Theorem 1]{andre}.

If $V=V_1\oplus V_2$ is non-simple, where $\dim_\Q V_i=2i$ for $i=1,2$, it follows from Lemma \ref{1dim} that the variation of $\Q$-Hodge structures $(p_1)_*(R_1f_*\mathbb{Q})$ induced by projection onto the factor $V_i$ is trivial. The possibilities for the endomorphism algebra $L_2$ of $V_2$ are completely described by Proposition \ref{prop non-simple case}.  Moreover, because $(p_1)_*(R_1f_*\mathbb{Q})$ is isotrivial, Schmid's rigidity theorem \cite[Theorem 7.24]{schmid} implies the monodromy representation $\overline{\rho}\colon \pi_1(X')\rightarrow GL(V)$ acts trivially on $V_1$. Namely, the representation $\overline{\rho}$ is determined solely by its action on $V_2$, and hence, using Theorem \ref{andrethm}, the connected monodromy group $T$ of the fibration $f$ is a normal subgroup of the commutator subgroup of the Hodge group of $V_2$. 

Since $V_2$ is a simple $4$-dimensional polarizable $\Q$-Hodge structures of type $\{(-1,0), (0,-1)\}$, Ribet's results\cite{ribet1} once again give that the Hodge group of $V_2$ is the connected component of the centralizer of $L_2$ in $Sp(V_2)$. If $L_2=\Q$, we thus have $Hg(V_2)=Sp(4)$ and if $L_2$ is a totally real quadratic field, then $Hg(V_2)= R_{L_2/\mathbb{Q}}SL(_{L_2}V)$. Since both groups are normal, the second half of the result follows once again from Theorem \ref{andrethm}.

The final part of the theorem dealing with cases arising from genus $3$ general complete intersection Kodaira fibrations follows from the characterizations of the endomorphism algebras found in Propositions \ref{simple case} and \ref{prop non-simple case}.
\end{proof}

\textbf{Acknowlegements.} The author would like to thank Donu Arapura, Don Blasius, Sebasti\'{a}n Reyes Carocca, Matt Kerr, Chad Schoen, Domingo Toledo, Kang Zuo, and especially Burt Totaro for helpful discussions in the preparation of this article. The author also gratefully acknowledges the support of the National Science Foundation through awards DGE-1144087 and DMS-1645877.

\bibliography{Kodaira.bib}
\bibliographystyle{alpha}

\end{document}